\documentclass[a4paper, 12pt]{article}
\pdfoutput=1

\usepackage[T1]{fontenc}
\usepackage[utf8]{inputenc}

\usepackage{amsfonts,amsmath,amssymb,mathtools,dsfont,microtype} 
\usepackage[dvipsnames]{xcolor} 
\usepackage[noBBpl]{mathpazo} 

\DeclareFontFamily{U} {cmr}{}

\DeclareFontShape{U}{cmr}{m}{n}{
	<-6> cmr5
	<6-7> cmr6
	<7-8> cmr7
	<8-9> cmr8
	<9-10> cmr9
	<10-12> cmr10
	<12-> cmr12}{}

\DeclareSymbolFont{Xcmr} {U} {cmr}{m}{n}
\DeclareMathSymbol{\Delta}{\mathord}{Xcmr}{'001}
\DeclareMathSymbol{\Upsilon}{\mathord}{Xcmr}{'007}
\DeclareMathSymbol{\Omega}{\mathord}{Xcmr}{'012}

\usepackage[labelsep = period, labelfont = bf, justification = centering]{caption}
\usepackage{float,graphicx,subcaption}
\DeclareCaptionSubType*[roman]{figure}

\usepackage{booktabs,makecell}
\usepackage{enumitem,mdwlist}
\setlist[itemize]{topsep=0ex,itemsep=0ex,parsep=0.4ex}
\setlist[enumerate]{topsep=0ex,itemsep=0ex,parsep=0.4ex}

\usepackage{pgf,tikz,ifthen,calc}
\usepackage{float, graphicx}
\usepackage{tkz-euclide}
\tkzSetUpPoint[fill=black, size = 3pt]
\tkzSetUpLine[color=black, line width=0.6pt]
\tkzSetUpCircle[color=black, line width=0.6pt]

\usepackage[left = 2.5cm, right = 2.5cm, top = 2.5cm, bottom = 2.5cm, headsep = 12pt, headheight = 15pt]{geometry}
\usepackage{parskip}

\usepackage[hyphens]{url} 
\usepackage[linktoc = all, hidelinks, colorlinks, unicode=true, pagebackref=true]{hyperref} 
\usepackage[capitalise, compress, nameinlink, noabbrev]{cleveref} 

\hypersetup{linkcolor={blue!70!black}, citecolor={green!70!black}, urlcolor={blue!70!black}}

\usepackage{amsthm,thmtools,thm-restate}
\declaretheorem[name = Theorem, numberwithin = section, style = plain]{thm}

\declaretheorem[name = Corollary, numberlike = thm, style = plain]{cor}
\declaretheorem[name = Conjecture, numberlike = thm, style = plain]{conj}

\declaretheorem[name = Lemma, numberlike = thm, style = plain]{lem}
\declaretheorem[name = Observation, numberlike = thm, style = plain]{obs}

\crefname{defi}{Definition}{Definitions}
\crefname{thm}{Theorem}{Theorems}
\crefname{lem}{Lemma}{Lemmas}
\crefname{conj}{Conjecture}{Conjectures}
\crefname{claim}{Claim}{Claims}
\crefname{cor}{Corollary}{Corollaries}
\crefname{obs}{Observation}{Observations}
\crefname{prop}{Proposition}{Propositions}
\crefname{que}{Question}{Questions}
\crefname{rem}{Remark}{Remarks}
\crefname{subsection}{\S}{\S\S}

\DeclareFontFamily{U}{matha}{\hyphenchar\font45}
\DeclareFontShape{U}{matha}{m}{n}{
	<5> <6> <7> <8> <9> <10> gen * matha
	<10.95> matha10 <12> <14.4> <17.28> <20.74> <24.88> matha12
}{}
\DeclareSymbolFont{matha}{U}{matha}{m}{n}

\DeclareMathSymbol{\specialuparrow}{\mathrel}{matha}{"D2}
\DeclareMathSymbol{\specialrightarrow}{\mathrel}{matha}{"D1}

\renewcommand*{\backref}[1]{}
\renewcommand*{\backrefalt}[4]{
	\ifcase #1 Not cited.%
	\or $\specialuparrow$#2%
	\else $\specialuparrow$#2%
	\fi%
}

\renewcommand{\epsilon}{\varepsilon}
\renewcommand{\ge}{\geqslant}
\renewcommand{\le}{\leqslant}

\renewcommand{\emptyset}{\varnothing}

\title{A constant-factor step towards Vizing's conjecture}
\author{Raphael Steiner\footnotemark[1]}
\date{\today}

\begin{document}

\maketitle

\renewcommand{\thefootnote}{\fnsymbol{footnote}} 

\footnotetext[1]{Department of Mathematics, ETH Z\"{u}rich, Switzerland (\textsf{\href{mailto:raphaelmario.steiner@math.ethz.ch}{raphaelmario.steiner@math.ethz.ch}}). Research supported by the SNSF Ambizione Grant No. 216071.}

\renewcommand{\thefootnote}{\arabic{footnote}} 

\begin{abstract}
\noindent Vizing's conjecture from 1963, considered by many the most important open problem in the field of graph domination, states that all graphs $G$ and $H$ satisfy $$\gamma(G\square H)\ge \gamma(G)\gamma(H),$$ where $\gamma$ denotes the domination number and $\square$ the Cartesian product. 

In a seminal result, Clark and Suen (2000) proved an approximate form of the conjecture, namely that $\gamma(G\square H)\ge \frac{1}{2}\gamma(G)\gamma(H)$ for all graphs $G$ and $H$. Despite several lower-order improvements of this bound and improvements for special classes of graphs $G$ and $H$, no absolute constant $c>\frac{1}{2}$ such that $\gamma(G\square H)\ge c\gamma(G)\gamma(H)$ for all graphs $G$ and $H$, has been known thus far.

\noindent In this paper, we obtain the first constant-factor improvement of the Clark-Suen bound by proving that for all graphs $G$ and $H$, we have
$$\gamma(G\square H)\ge c\gamma(G)\gamma(H),$$
where
$$c=\frac{5+\sqrt{73}}{24}\approx 0.5643.$$
Along the way, we prove another lower bound on $\gamma(G\square H)$ which outperforms the above bound for many graphs.
\end{abstract}

\section{Introduction}
The \emph{domination number} $\gamma(G)$ of a graph $G$ is amongst the most fundamental graph parameters and has been studied since the beginnings of graph theory. It is defined as the minimum size of a subset of vertices of $G$ such that every vertex is either contained in the set or has a neighbor in it. The theory of graph domination has grown into a rich and flourishing branch of modern graph theory, yet several fundamental problems remain. Vizing's conjecture, dating back to the early 60s, is perhaps the most famous among those open problems, and claims that the domination number is super-multiplicative with respect to the Cartesian product\footnote{Recall that given two graphs $G$ and $H$, their \emph{Cartesian product} $G\square H$ is defined to be the graph with vertex set $V(G)\times V(H)$, where two distinct vertices $(g,h)$ and $(g',h')$ are adjacent if and only if either $gg'\in E(G)$ and $h=h'$, or $g=g'$ and $hh'\in E(H)$.}.

\begin{conj}[Vizing 1963~\cite{MR209178}, see also~\cite{MR240000}]\label{con:vizing}
For all graphs $G$ and $H$, we have
$$\gamma(G\square H)\ge \gamma(G)\gamma(H).$$
\end{conj}
Vizing's conjecture has attracted a tremendous amount of attention from graph theory researchers in the past, and it would be impossible to even come close to a full account of known partial results on the conjecture and its variants here. Instead, we will refer to the textbooks and survey articles~\cite{MR4607811,MR1605684,MR2864622} for a detailed coverage of the history of progress on the conjecture, and focus here only on the directly relevant previous work.

The probably most outstanding partial result towards Vizing's conjecture that applies to \emph{all} graphs $G$ and $H$ rather than such in special classes, is due to Clark and Suen from 2000~\cite{MR1763970}, proving that 
$$\gamma(G\square H)\ge \frac{1}{2}\gamma(G)\gamma(H)$$ for all graphs $G$ and $H$, that is, the conjecture holds up to a multiplicative factor of $2$. Clark and Suen's result was subsequently improved further, first by Suen and Tarr in 2012~\cite{MR2880639} to 
$$\gamma(G\square H)\ge \frac{1}{2}\gamma(G)\gamma(H)+\frac{1}{2}\min\{\gamma(G),\gamma(H)\}$$ and then further by Zerbib in 2019~\cite{MR4035865} to 
$$\gamma(G\square H)\ge \frac{1}{2}\gamma(G)\gamma(H)+\frac{1}{2}\max\{\gamma(G),\gamma(H)\}.$$
While these improvements are significant when one of $\gamma(G),\gamma(H)$ is small, asymptotically, that is, when both $\gamma(G)$ and $\gamma(H)$ grow large, they do not improve the Clark-Suen bound by a constant factor. Such a constant-factor improvement, namely a proof that for some absolute constant $c>\frac{1}{2}$, every two graphs $G$ and $H$ satisfy $$\gamma(G\square H)\ge c\gamma(G)\gamma(H),$$ has remained open thus far and has been highlighted in several previous works as an open problem. A noteworthy recent advance in this direction (but no full solution) is due to Bre\v{s}ar~\cite{bresar}, who proved the following result. Here, $\rho(G)$ denotes the so-called \emph{$2$-packing number} of the graph $G$, defined as the maximum size of a subset $P$ of vertices such that the closed neighborhoods of any two vertices in $P$ are disjoint (such a set is called a \emph{$2$-packing}). Observe that every graph $G$ satisfies $\rho(G)\le \gamma(G)$, however in general $\rho(G)$ can be much smaller than $\gamma(G)$.
\begin{thm}[Bre\v{s}ar, cf.~Theorem~3 in~\cite{bresar}]\label{thm:bresar}
For all graphs $G$ and $H$, we have
$$\gamma(G\square H)\ge \max\left\{\frac{2\gamma(G)-\rho(G)}{3}\gamma(H),\frac{2\gamma(H)-\rho(H)}{3}\gamma(G)\right\}.$$
\end{thm}
In particular, this result implies that one can improve the Clark-Suen bound by a constant factor in the special case when the packing number $\rho(G)$ is significantly smaller than half of the domination number $\gamma(G)$.

As the main result of this paper, we provide the first constant-factor improvement of the Clark-Suen bound from 2000 that is valid for all graphs $G$ and $H$, without any additional assumptions.
\begin{thm}\label{thm:main}
For all graphs $G$ and $H$, we have
$$\gamma(G\square H)\ge c\gamma(G)\gamma(H),$$ where $c:=\frac{5+\sqrt{73}}{24}\approx 0.5643$.
\end{thm}

A key new ingredient of our proof of Theorem~\ref{thm:main} is a lower bound on the domination number of $G\square H$ which is of a similar type as that of Bre\v{s}ar stated in Theorem~\ref{thm:bresar}, but quantitatively better in certain (for us relevant) ranges of the parameters. To state this bound, we first need to introduce the following important definitions (cf.~\cite{MR4180634}): Given $k\in \mathbb{Z}_{\ge 0}$ and a graph $G$, a non-negative integer weight-function $f:V(G)\rightarrow \mathbb{Z}_{\ge 0}$ is called a \emph{$k$-packing ($k$-dominating) function} if the total weight of every closed neighborhood in $G$ is at most (at least) $k$. We denote by $\rho^{\{k\}}(G)$ (and $\gamma^{\{k\}}(G)$, respectively), the maximum (minimum) possible total weight of a $k$-packing ($k$-dominating) weight function on $G$. It is easy to see from these definitions that $\rho^{\{1\}}(G)=\rho(G)$ and $\gamma^{\{1\}}(G)=\gamma(G)$. It is also immediate from the definitions that $\rho^{\{k\}}(G)\ge k\rho(G)$ and $\gamma^{\{k\}}(G)\le k\gamma(G)$ for every $k\in \mathbb{Z}_{\ge 0}$.

With these important notions at hand, we can now state our key new lower bound on $\gamma(G\square H)$.
\begin{thm}\label{thm:Vizing2packing}
For all graphs $G$ and $H$, we have
$$\gamma(G\square H)\ge \max\left\{\frac{3\gamma(G)-\rho^{\{2\}}(G)}{4}\gamma(H),\frac{3\gamma(H)-\rho^{\{2\}}(H)}{4}\gamma(G)\right\}.$$
\end{thm}
In particular, note that this theorem gives a significant improvement over the constant in Theorem~\ref{thm:main} whenever one of $G$ or $H$ satisfies that $\rho^{\{2\}}(\cdot)$ is small relative to $\gamma(\cdot)$.

To prove our main result, Theorem~\ref{thm:main}, we combine Theorem~\ref{thm:Vizing2packing} with the aforementioned Theorem~\ref{thm:bresar} as well as with another lower bound on $\gamma(G\square H)$ which is a consequence of a result due to Hou and Lu~\cite{houlu}, which is stated in Section~\ref{sec:proof}.

\paragraph*{Notation and Terminology.} We summarize some important notation and terminology used throughout the paper. All graphs considered in this paper are simple, finite and have at least one vertex. Given a graph $G$, we denote by $V(G)$ its vertex set and by $E(G)$ its edge set. Given a vertex $v\in V(G)$, we denote by $N_G[v]$ its \emph{closed neighborhood}, i.e. the set consisting of $v$ and all its neighbors. We say that a vertex $v\in V(G)$ is \emph{dominated} by a subset $D\subseteq V(G)$ of vertices (or equivalently, that $D$ \emph{dominates} $v$) if $N[v]\cap D\neq \emptyset$, and we say that $D$ \emph{dominates} a subset $S\subseteq V(G)$ if every member of $S$ is dominated by $D$. We denote the minimum size of a set of vertices in $G$ which dominates $S$ as $\gamma_G(S)$. In particular, we have $\gamma_G(V(G))=\gamma(G)$ for every graph $G$. We also use $\rho_G(S)$ to denote the maximum size of a subset of $S$ which forms a $2$-packing in $G$ (i.e., a subset of vertices such that any two vertices in this subset have disjoint closed neighborhoods). Given a weight function $f:V(G)\rightarrow \mathbb{Z}_{\ge 0}$ on the vertices of a graph $G$ and a subset $S\subseteq V(G)$, we denote by $f(S):=\sum_{x\in S}f(x)$ the \emph{total weight} of $S$. 

\paragraph*{Organization.} In Section~\ref{sec:aux} we prove and collect several important auxiliary results for our proof of Theorem~\ref{thm:Vizing2packing}, in particular, we prove a novel upper bound on $\gamma_G(S)$ for subsets $S$ of vertices in a graph $G$ in terms of $\rho^{\{2\}}(G)$ and the size of the subset $S$, which may be of independent interest. In Section~\ref{sec:proof} we then give the proof of Theorem~\ref{thm:Vizing2packing} by combining the new ingredients from Section~\ref{sec:aux} with the same proof approach used by Bre\v{s}ar~\cite{bresar} to prove Theorem~\ref{thm:bresar}. Afterwards, we derive a lower bound on $\gamma(G\square H)$ involving $\rho^{\{2\}}(G), \rho^{\{2\}}(H)$ from a result by Hou and Lu~\cite{houlu}. Finally, we combine this bound with the bounds given by Theorems~\ref{thm:bresar} and~\ref{thm:Vizing2packing} to prove our main result, Theorem~\ref{thm:main}.

\paragraph*{AI Disclosure.} We used ChatGPT 5.5 Pro to assist us with the search for literature surrounding Vizing's conjecture, in particular providing us with the important references~\cite{houlu,bresar} which are crucial ingredients in our proof. We also used it to help us check this document for misprints and typos.

\section{Auxiliary results}\label{sec:aux}
In this section, we prove a new upper bound on $\gamma_G(S)$ in terms of the size of the subset $S$ of vertices and $\rho^{\{2\}}(G)$ (Lemma~\ref{lem:cor} below). We prepare its proof with two different such upper bounds, from which Lemma~\ref{lem:cor} can then be derived. While the first (Lemma~\ref{lem:key} below) is novel, the second (Lemma~\ref{lem:obs} below) already appears in slightly weaker form (replacing $\gamma_G(S)$ by $\gamma(G)$) in the work of Bre\v{s}ar.
\begin{lem}\label{lem:key}
Let $G$ be a graph and $S\subseteq V(G)$. If $|N_G[v]\cap S|\le 2$ for every $v\in V(G)$, then we have
$$\gamma_G(S)\le \rho^{\{2\}}(G)-\rho_G(S).$$
\end{lem}
\begin{proof}
Let $P\subseteq S$ be a $2$-packing contained in $S$ of maximum size, i.e. $|P|=\rho_G(S)$. We now define an auxiliary bipartite graph $B$ with vertex set $V(B)=S$ and bipartition $V(B)=(S\setminus P)\sqcup P$ as follows: Given $s\in S\setminus P$ and $p\in P$, we make $s$ and $p$ adjacent in $B$ if and only if $N_G[s]\cap N_G[p]\neq \emptyset$. Now, consider a maximum matching $M=\{s_1p_1,\ldots,s_tp_t\}$ in $B$. Since $B$ is a bipartite graph, we can apply K\"{o}nig's theorem to find that $|M|=t$ equals the minimum size of a vertex cover in $B$. Since the vertex covers of a graph are exactly the complements of its independent sets, it now follows that there exists an independent set $I$ in $B$ of size $|I|=|V(B)|-t=|S|-t$. 

Next, let us pick, for each $i\in [t]$, some vertex $d_i\in N_G[s_i]\cap N_G[p_i]$. We can then see that the set $(S\setminus \{s_1,p_1,\ldots,s_t,p_t\})\cup \bigcup_{i=1}^t\{d_i\}$ dominates $S$ in $G$ and has size at most $|S|-2t+t\le |S|-t$. Hence, it follows that $\gamma_G(S)\le |S|-t=|I|$. 

Finally, let us define a non-negative integer weight function $f:V(G)\rightarrow \mathbb{Z}_{\ge 0}$ as follows. For every $x\in V(G)$, we set
$$f(x):=\begin{cases}
2, & \text{ if }x\in P\cap I,\\
1, & \text{ if }x\in (P\cup I)\setminus (P\cap I), \\
0, & \text{ if } x\notin P\cup I.
\end{cases}$$
It then follows directly from this definition that $\sum_{x\in V(G)}f(x)=|P|+|I|$. We next claim that $f(N_G[v])\le 2$ for every $v\in V(G)$. So consider any $v\in V(G)$. Note that by our assumption in the lemma, we then have $|N_G[v]\cap S|\le 2$. Since $f$ is only supported on $P\cup I\subseteq S$ and does not take values larger than $2$, the claim trivially holds if $|N_G[v]\cap (P\cup I)|\le 1$, so we may suppose that there are two distinct elements $v_1, v_2\in N_G[v]\cap (P\cup I)$, and these must then be the only elements of $N_G[v]\cap S$, so that $f(N_G[v])=f(v_1)+f(v_2)$. If both $v_1,v_2\in (P\cup I)\setminus (P\cap I)$, then by definition of $f$ we have $f(v_1)+f(v_2)=1+1\le 2$, and so the claim holds. It thus only remains to consider (up to relabelling) the case  $v_1\in P\cup I, v_2\in P\cap I$. We now claim that in fact, this case cannot occur. Indeed, recalling that $P$ is a $2$-packing, in this case we could not have $v_1\in P$, and so in fact, we would need to have $v_1\in I\setminus P$. Since $N[v_1]\cap N[v_2]$ contains $v$ and is hence non-empty, it then follows by our definition of $B$ that $v_1v_2\in E(B)$. However, since $v_1, v_2\in I$, this would be a contradiction to the fact that $I$ is an independent set in $B$. Hence, indeed this last case is impossible, showing that we have $f(N_G[v])\le 2$ for all $v\in V(G)$. Given this and using the definition of $\rho^{\{2\}}(G)$, it now follows that
$$\rho^{\{2\}}(G)\ge \sum_{x\in V(G)}f(x)=|P|+|I|\ge \rho_G(S)+\gamma_G(S).$$ Rearranging now yields the inequality claimed in the lemma, concluding its proof.
\end{proof}
As mentioned before, we next give a simple bound on $\gamma_G(S)$ in terms of $|S|$ and $\rho_G(S)$.
\begin{lem}\label{lem:obs}
For every graph $G$ and every $S\subseteq V(G)$, we have
$$\gamma_G(S)\le \frac{|S|+\rho_G(S)}{2}.$$
\end{lem}
\begin{proof}
Define an auxiliary graph $B$ with $V(B)=S$, where two distinct vertices $s_1,s_2\in S$ form an edge if and only if $N_G[s_1]\cap N_G[s_2]\neq \emptyset$. Let $M=\{s_1s_1',\ldots,s_ts_t'\}$ be a maximum matching in $B$. For each $i\in [t]$, pick some vertex $d_i\in N_G[s_i]\cap N_G[s_i']$. Then the set $(S\setminus \{s_1,s_1',\ldots,s_t,s_t'\})\cup \bigcup_{i=1}^t\{d_i\}$ dominates $S$ in $G$, and hence we have $\gamma_G(S)\le |S|-2t+t=|S|-t$. Note further that $S\setminus \{s_1,s_1',\ldots,s_t,s_t'\}$ is an independent set in $B$, and hence a $2$-packing in $G$. In particular, we find that $|S|-2t\le \rho_G(S)$. Rearranging yields $t\ge \frac{|S|-\rho_G(S)}{2}$. Plugging this into our above upper bound on $\gamma_G(S)$ now yields
$$\gamma_G(S)\le |S|-t\le |S|-\frac{|S|-\rho_G(S)}{2}=\frac{|S|+\rho_G(S)}{2},$$ as desired. This concludes the proof of the lemma.
\end{proof}
Finally, we combine the previous two lemmas with an inductive proof to obtain the following new upper bound on $\gamma_G(S)$, which will be key in our proof of Theorem~\ref{thm:Vizing2packing} in the next section.
\begin{lem}\label{lem:cor}
For every graph $G$ and every $S\subseteq V(G)$, we have
$$\gamma_G(S)\le \frac{|S|+\rho^{\{2\}}(G)}{3}.$$
\end{lem}
\begin{proof}
We prove the claim for the fixed graph $G$ by induction on $|S|$. It trivially holds when $|S|=0$ since then $\gamma_G(S)=\gamma_G(\emptyset)=0$. For the induction step, suppose we are given a non-empty subset $S\subseteq V(G)$ and that we already proved $\gamma_G(S')\le \frac{|S'|+\rho^{\{2\}}(G)}{3}$ for all $S'\subseteq V(G)$ with $|S'|<|S|$. Our goal in the following is to show that $\gamma_G(S)\le \frac{|S|+\rho^{\{2\}}(G)}{3}$.

Suppose first that there exists some vertex $v\in V(G)$ with $|N_G[v]\cap S|\ge 3$. Then define $S':=S\setminus N_G[v]$ and observe that $|S'|\le |S|-3$ as well as $\gamma_G(S)\le \gamma_G(S')+1$. By the inductive assumption, we now find 
$$\gamma_G(S)\le \gamma_G(S')+1\le \frac{|S'|+\rho^{\{2\}}(G)}{3}+1\le \frac{|S|+\rho^{\{2\}}(G)}{3},$$ concluding the proof in this case.

Hence, moving on we may assume that $|N_G[v]\cap S|\le 2$ for every $v\in V(G)$. We may now apply Lemma~\ref{lem:key} to find that $\gamma_G(S)\le \rho^{\{2\}}(G)-\rho_G(S)$. Furthermore, by Lemma~\ref{lem:obs} we have $\gamma_G(S)\le \frac{|S|+\rho_G(S)}{2}$. Taking a convex combination of these two inequalities, we find 
$$\gamma_G(S)\le \frac{1}{3}\cdot(\rho^{\{2\}}(G)-\rho_G(S))+\frac{2}{3}\cdot \frac{|S|+\rho_G(S)}{2}=\frac{|S|+\rho^{\{2\}}(G)}{3},$$ as desired. This concludes the proof of the lemma.
\end{proof}

\section{Proofs of the main results}\label{sec:proof}
In this section, we prove our main results, namely Theorems~\ref{thm:Vizing2packing} and~\ref{thm:main}. 

We start with the proof of Theorem~\ref{thm:Vizing2packing}. Its structure follows closely that of Bre\v{s}ar in~\cite{bresar}, but the key difference lies in using the new upper bound $\gamma_G(S)\le \frac{|S|+\rho^{\{2\}}(G)}{3}$ from Lemma~\ref{lem:cor} at the relevant point, rather than the bound $\gamma_G(S)\le \frac{|S|+\rho(G)}{2}$ used by Bre\v{s}ar in~\cite{bresar}. 

\begin{proof}[Proof of Theorem~\ref{thm:Vizing2packing}]
Let $G$ and $H$ be given arbitrarily. Since $G\square H$ and $H\square G$ are isomorphic and thus have the same domination number, to establish the theorem it is sufficient to prove the inequality 
$$\gamma(G\square H)\ge \frac{3\gamma(G)-\rho^{\{2\}}(G)}{4}\gamma(H).$$

In the remainder of the proof, let us fix $D$ to be a minimum dominating set of $G\square H$. In the following, we closely follow the terminology and structure of the proof of Theorem~\ref{thm:Vizing2packing} in~\cite{bresar}. In particular, for a subset $X\subseteq V(G)\times V(H)$ we denote by $P_G(X)\subseteq V(G), P_H(X)\subseteq V(H)$ its projections to the first and second coordinates, respectively.

We start by considering a minimum dominating set $\{h_1,\ldots,h_k\}$ for $H$ and throughout the remainder of the proof fix a partition of $V(H)$ into subsets $\pi_1,\ldots,\pi_k$ such that $\pi_i\subseteq N_H[h_i]$ for each $i$ (evidently, such a partition exists).

For $i=1,\ldots,k$ define $G_i:=V(G)\times \pi_i$ and $D_i:=D\cap G_i$. Given some $g\in V(G)$, the set $C_g^i:=\{g\}\times \pi_i$ is called a \emph{cell}. We say that a cell $C_g^i$ is \emph{vertically dominated} by $D$ if each of its vertices has a neighbor (in $G\square H$) within the set $D\cap (\{g\}\times (V(H)\setminus \pi_i))$, otherwise it is called \emph{vertically undominated}.

For every $i\in [k]$, let us define $L_i:=\{g\in V(G)|C_g^i\text{ is vertically dominated}\}.$

Pause to note that by definition of ``vertically undominated'', for every vertex $g\in V(G)\setminus L_i$, we have that there exists some $h\in \pi_i$ such that the vertex $(g,h)\in C_g^i$ does not have a neighbor in $D\cap (\{g\}\times (V(H)\setminus \pi_i))$. Since $D$ is a dominating set, this implies that either the set $C_g^i$ contains a neighbor of $(g,h)$ in $D$ (which is then also contained in $D_i$), or that there exists some $g'\in N_G[g]$ such that $(g',h)\in D$ (which is then, as well, contained in $D_i$). In each case, we find that $g$ is contained in the closed neighborhood (in $G$) of some vertex in $P_G(D_i)$. This implies that, for each $i\in [k]$, we have that $\gamma_G(V(G)\setminus L_i)\le |P_G(D_i)|\le |D_i|$.

Furthermore, by Lemma~\ref{lem:cor} we have that $\gamma_G(L_i)\le \frac{|L_i|+\rho^{\{2\}}(G)}{3}$. Combining these bounds, we find that
$$\gamma(G)\le \gamma_G(V(G)\setminus L_i)+\gamma_G(L_i)\le |D_i|+\frac{|L_i|+\rho^{\{2\}}(G)}{3}$$ for all $i\in [k]$. Summing this inequality over all $i$ yields
$$\gamma(G)\gamma(H)=k\gamma(G)\le \sum_{i=1}^k\left(|D_i|+\frac{|L_i|+\rho^{\{2\}}(G)}{3}\right)=|D|+\frac{1}{3}\sum_{i=1}^{k}|L_i|+\frac{\rho^{\{2\}}(G)\gamma(H)}{3}.$$
Next, for each $g\in V(G)$, let us denote by $m_g$ the number of all vertically dominated cells $C_g^i$ where $i$ ranges in $[k]$. Clearly,  $\sum_{i=1}^k|L_i|=\sum_{g\in V(G)}m_g$, since both sums count the total number of vertically dominated cells. 

Now fix any $g\in V(G)$. By definition of ``vertically dominated'', we then have that $P_H(D\cap (\{g\}\times V(H)))$ dominates (in $H$) all the sets $\pi_i$ such that $C_g^i$ is vertically dominated. Furthermore, for each of the $k-m_g$ indices $i\in [k]$ such that $C_g^i$ is not vertically dominated, we can dominate (in $H$) all vertices in $\pi_i$ using the vertex $h_i$. All in all, this implies that $$\gamma(H)\le |P_H(D\cap (\{g\}\times V(H)))|+k-m_g\le |D\cap (\{g\}\times V(H))|+k-m_g.$$  Since $k=\gamma(H)$, we have shown that $m_g\le |D\cap (\{g\}\times V(H))|$ for all $g\in V(G)$. Summing this over all $g\in V(G)$ now yields that 
$$\sum_{i=1}^k|L_i|=\sum_{g\in V(G)}m_g\le |D|.$$
Let us plug this estimate into our above upper bound on $\gamma(G)\gamma(H)$. We then obtain:
$$\gamma(G)\gamma(H)\le |D|+\frac{1}{3}|D|+\frac{\rho^{\{2\}}(G)\gamma(H)}{3}.$$ Rearranging this yields
$$\gamma(G\square H)=|D|\ge \frac{3}{4}\left(\gamma(G)\gamma(H)-\frac{\rho^{\{2\}}(G)\gamma(H)}{3}\right)=\frac{3\gamma(G)-\rho^{\{2\}}(G)}{4}\gamma(H),$$ as desired. This concludes the proof of the theorem.
\end{proof}
Before giving the proof of our main result, Theorem~\ref{thm:main}, we need to derive another important lower bound on $\gamma(G\square H)$ involving $\rho^{\{2\}}(G)$ and $\rho^{\{2\}}(H)$ (stated in Corollary~\ref{cor:secondlowerbound} below). This bound will be a consequence of the following, stronger, bound, shown by Hou and Lu.
\begin{thm}[Hou and Lu, cf. equation~(5) in~\cite{houlu}]\label{thm:houlu}
For all graphs $G$ and $H$, we have
$$\gamma(G\square H)\ge \rho(G)\gamma(H)+\gamma^{\{m\}}(H),$$
where $m:=\gamma(G)-\rho(G)$.
\end{thm}
Before stating the relevant corollary, let us make the following simple observation.
\begin{obs}\label{obs:obvious}
For every graph $H$ and every $m\in \mathbb{Z}_{\ge 0}$, we have $\gamma^{\{m\}}(H)\ge \frac{\rho^{\{2\}}(H)}{2}m$.
\end{obs}
\begin{proof}
To start with, recall that by definition of $\rho^{\{2\}}(H)$ and $\gamma^{\{m\}}(H)$, there exist functions $f,g:V(H)\rightarrow \mathbb{Z}_{\ge 0}$ such that $f(N_H[v])\le 2$ for every $v\in V(H)$ and $\sum_{v\in V(H)}f(v)=\rho^{\{2\}}(H)$, as well as $g(N_H[v])\ge m$ for every $v\in V(H)$ and $\sum_{v\in V(H)}g(v)=\gamma^{\{m\}}(H)$. 

Using these inequalities, we now find
\begin{align*}\gamma^{\{m\}}(H)&=\sum_{v\in V(H)}g(v)
\ge \frac{1}{2}\sum_{v\in V(H)}g(v)f(N_H[v])\\
&=\frac{1}{2}\sum_{v\in V(H)}\sum_{w\in N_H[v]}g(v)f(w)\\
&=\frac{1}{2}\sum_{w\in V(H)}\sum_{v\in N_H[w]}g(v)f(w)\\
&=\frac{1}{2}\sum_{w\in V(H)}f(w)g(N_H[w])\\
&\ge \frac{m}{2}\sum_{w\in V(H)}f(w)=\frac{m}{2}\rho^{\{2\}}(H)=\frac{\rho^{\{2\}}(H)}{2}m,\end{align*} as desired. This concludes the proof.
\end{proof}

The following lower bound on $\gamma(G\square H)$ is now an immediate consequence of Theorem~\ref{thm:houlu} and Observation~\ref{obs:obvious}. 
\begin{cor}\label{cor:secondlowerbound}
For all graphs $G$ and $H$, we have 
\begin{align*}&\gamma(G\square H)\ge\\ &\max\left\{\rho(G)\gamma(H)+\frac{\rho^{\{2\}}(H)}{2}(\gamma(G)-\rho(G)),\rho(H)\gamma(G)+\frac{\rho^{\{2\}}(G)}{2}(\gamma(H)-\rho(H))\right\}.\end{align*}
\end{cor}

\begin{proof}
By symmetry, it suffices to show $\gamma(G\square H)\ge \rho(G)\gamma(H)+\frac{\rho^{\{2\}}(H)}{2}(\gamma(G)-\rho(G))$. This follows directly by combining Theorem~\ref{thm:houlu} and Observation~\ref{obs:obvious}.
\end{proof}
With all ingredients set in place, we are now finally ready to combine Theorems~\ref{thm:bresar},~\ref{thm:Vizing2packing} and Corollary~\ref{cor:secondlowerbound} to prove Theorem~\ref{thm:main}.
\begin{proof}[Proof of Theorem~\ref{thm:main}]
Let $G$ and $H$ be any given graphs. Using Observation~\ref{obs:obvious}, we find $$1\le \rho(G)\le \frac{\rho^{\{2\}}(G)}{2}\le \gamma(G), 1\le \rho(H)\le \frac{\rho^{\{2\}}(H)}{2}\le \gamma(H).$$ Let us now define $$x_1:=\frac{\rho(G)}{\gamma(G)}, x_2:=\frac{\rho^{\{2\}}(G)}{2\gamma(G)}, y_1:=\frac{\rho(H)}{\gamma(H)}, y_2:=\frac{\rho^{\{2\}}(H)}{2\gamma(H)}.$$
Observe that by the above, we have that $0<x_1\le x_2\le 1$ and $0<y_1\le y_2\le 1$. We now give several lower bounds on the ratio $\frac{\gamma(G\square H)}{\gamma(G)\gamma(H)}$. 

First, note that by Theorem~\ref{thm:bresar} we get 
$$\frac{\gamma(G\square H)}{\gamma(G)\gamma(H)}\ge \max\left\{\frac{2-x_1}{3},\frac{2-y_1}{3}\right\}.$$

Second, pause to verify that by Theorem~\ref{thm:Vizing2packing} we have
$$\frac{\gamma(G\square H)}{\gamma(G)\gamma(H)}\ge \max\left\{\frac{3}{4}-\frac{x_2}{2},\frac{3}{4}-\frac{y_2}{2}\right\}.$$

Third, from Corollary~\ref{cor:secondlowerbound}, we get that
$$\frac{\gamma(G\square H)}{\gamma(G)\gamma(H)}\ge \max\left\{x_1+y_2(1-x_1),y_1+x_2(1-y_1)\right\}.$$
Summarizing, to establish the claim of the theorem, it thus suffices to show that 
$$\max\left\{\frac{2-x_1}{3},\frac{2-y_1}{3},\frac{3}{4}-\frac{x_2}{2},\frac{3}{4}-\frac{y_2}{2},x_1+y_2(1-x_1),y_1+x_2(1-y_1)\right\}\ge c=\frac{5+\sqrt{73}}{24}$$ for all real numbers $x_1,x_2,y_1,y_2\in [0,1]$. 

If $x_1<\frac{11-\sqrt{73}}{8}$ or $y_1<\frac{11-\sqrt{73}}{8}$, then one of the first two terms in the above maximum exceeds $c$, and so we are done in this case. 

Similarly, if $x_2<\frac{13-\sqrt{73}}{12}$ or $y_2<\frac{13-\sqrt{73}}{12}$, then one of the second two terms in the above maximum exceeds $c$, and so we are done in this case, too.

Hence, it remains to consider the case that $x_1, y_1\ge \frac{11-\sqrt{73}}{8}$ and $x_2,y_2\ge \frac{13-\sqrt{73}}{12}$. Pause to note that the last two terms of the above maximum are both monotonically increasing in each of the four variables when they range within $[0,1]$. Hence, we then find that the above maximum is lower-bounded by
$$\frac{11-\sqrt{73}}{8}+\frac{13-\sqrt{73}}{12}\cdot \left(1-\frac{11-\sqrt{73}}{8}\right)$$
$$=\frac{5+\sqrt{73}}{24}=c,$$
as desired. This concludes the proof of the theorem.
\end{proof}
\fontsize{11pt}{12pt}
\selectfont
	
\hypersetup{linkcolor={red!70!black}}
\setlength{\parskip}{2pt plus 0.3ex minus 0.3ex}

\bibliographystyle{auxfile.bst}
\bibliography{bib.bib}

\end{document}